\documentclass[11pt]{article}
\usepackage[T1]{fontenc}
\usepackage[utf8]{inputenc}
\usepackage{lmodern}
\usepackage[a4paper,margin=3cm]{geometry}
\usepackage{amsmath,amssymb,amsthm,mathtools,booktabs}
\usepackage{microtype}
\usepackage{graphicx,tikz}
\usepackage[labelsep=period]{caption}
\usetikzlibrary{shapes.geometric}
\usepackage[hidelinks]{hyperref}
\hypersetup{pdftitle={A sharp density bound for 5-connected graphs with no K7= minor},pdfauthor={Caibing Chang, Zijian Deng, Qinfei Tang, Caihong Yang},pdfsubject={Graph minors and extremal graph theory}}
\newtheorem{theorem}{Theorem}[section]
\newtheorem{lemma}[theorem]{Lemma}
\theoremstyle{remark}
\newtheorem{remark}[theorem]{Remark}
\numberwithin{equation}{section}
\newcommand{\Ke}{K_7^{=}}
\newcommand{\W}{\mathcal W}
\newcommand{\T}{\mathcal T}
\newcommand{\Sc}{\mathcal S}
\newcommand{\dd}{\partial}
\newcommand{\needroom}[1]{\par\begingroup\skip0=\lastskip\vskip-\skip0\dimen0=#1\relax\advance\dimen0 by\skip0\vskip0pt plus\dimen0\penalty-100\vskip0pt plus-\dimen0\vskip\dimen0\penalty9999\vskip-\dimen0\vskip\skip0\endgroup}

\title{A sharp density bound for 5-connected graphs\\with no $\Ke$ minor}
\author{%
Caibing Chang\textsuperscript{1}\thanks{Email: \href{mailto:changcaibing2018@163.com}{changcaibing2018@163.com}.},\quad
Zijian Deng\textsuperscript{2}\thanks{Email: \href{mailto:zj1329205716@163.com}{zj1329205716@163.com}.},\quad
Qinfei Tang\textsuperscript{3}\thanks{Email: \href{mailto:tqf9500@126.com}{tqf9500@126.com} (corresponding author).}\quad
and~Caihong Yang\textsuperscript{4}\thanks{Email: \href{mailto:yangch@cug.edu.cn}{yangch@cug.edu.cn}.}\\[0.9em]
\small\textsuperscript{1}Faculty of Science\\
\small Guangdong University of Petrochemical Technology, Maoming, Guangdong 525000, China\\[0.4em]
\small\textsuperscript{2}School of Mathematics and Statistics\\
\small Lingnan Normal University, Zhanjiang, Guangdong 524048, China\\[0.4em]
\small\textsuperscript{3}Department of Mathematics Research\\
\small Fujian Institute of Education, Fuzhou, China\\[0.4em]
\small\textsuperscript{4}School of Mathematics and Physics\\
\small China University of Geosciences, Wuhan, China}
\date{}

\begin{document}
\maketitle
\begin{abstract}
Let $\Ke$ be obtained from $K_7$ by deleting two independent edges.
We prove that every 5-connected graph on $n\ge7$ vertices with at least $4n-9$ edges contains a $\Ke$ minor, settling Conjecture~1.4 of Dvo\v r\'ak, Norin and Rahman (arXiv preprint 2609.17760v1).
The bound is sharp. We prove the stronger statement that every $4$-bilight graph on $n\ge4$ vertices with at least $4n-9$ edges contains either a $\Ke$ minor or a $K_6$ subgraph.
Within their reduction framework, we strengthen the rooted-minor theorem. We show that every $4$-light 5-rooted graph of rooted $4$-density at least two has a model with two nonroot vertices and at most one missing edge incident with them.
At the critical density, reductions preserve density exactly, which prevents them from creating a new $K_6$ subgraph.

\medskip
\noindent\textbf{Keywords}\quad Graph minor, extremal graph theory, rooted minor, connectivity.

\smallskip
\noindent\textbf{MSC 2020}\quad 05C35, 05C83, 05C40.
\end{abstract}

\section{Introduction}
All graphs are finite and simple, and $e(G)=|E(G)|$. Extremal graph-minor theory asks how many edges force a prescribed graph as a minor. Complete targets have been central to this question since Mader's work~\cite{Mad68}. Thomason~\cite{Tho01} determined the asymptotic extremal function for large complete minors. For small complete targets, J{\o}rgensen~\cite{Jor94} determined the extremal function for $K_8$ minors, and Song and Thomas~\cite{ST06} did so for $K_9$ minors. For a nearly complete target, Song~\cite{Son05} determined the extremal function for $K_8^-$ minors.

Here the target is $\Ke$, obtained from $K_7$ by deleting two independent edges. It lies between $K_6$ and $K_7$ in the minor order, since contracting an edge joining endpoints of its two missing edges gives a $K_6$ minor. We write $K_7^-$ for deletion of a single edge and $K_7^{\vee}$ for deletion of two edges with a common end. We study the exact density threshold for $\Ke$ under 5-connectivity, where separations of order at most four are excluded.

Jakobsen's earlier extremal theorem concerns the family consisting of $\Ke$ and $K_7^{\vee}$. In the formulation recalled by Rolek~\cite{Rol20}, every graph on $n\ge6$ vertices with at least $4n-9$ edges either contains one of these two graphs as a minor or is a $(K_6,3)$-cockade. Such a cockade is obtained recursively from copies of $K_6$ by identifying triangles in two previously constructed graphs. For 5-connected graphs on at least seven vertices, the cockade exception is excluded, but the conclusion still allows either deletion pattern. Theorem~\ref{thm:main} guarantees the specified graph $\Ke$ at the same edge threshold.

Dvo\v r\'ak, Norin and Rahman~\cite[Theorem~1.3]{DNR26} proved that every 5-connected graph on $n\ge6$ vertices with at least $4n-7$ edges contains a $\Ke$ minor. Their density theorem shows that $\Ke$-minor-free graphs are 6-colourable. They conjectured that the threshold could be lowered to $4n-9$ when $n\ge7$. We prove this conjecture.
\begin{theorem}\label{thm:main}
Every 5-connected graph $G$ with $n\ge7$ vertices and $e(G)\ge4n-9$ contains $\Ke$ as a minor.
\end{theorem}
We use the notion of $4$-bilightness introduced in~\cite{DNR26}. For a nonempty set $Y\subseteq V(G)$, put
\[
\dd_G Y=N_G(Y)\setminus Y,\qquad \rho_4(G,Y)=e(G[Y])+e_G(Y,\dd_GY)-4|Y|.
\]
The graph $G$ is \emph{$4$-bilight} if there do not exist disjoint, nonadjacent, nonempty vertex sets $Y,Z$ such that
\[
|\dd_GY|,|\dd_GZ|\le4,\qquad \rho_4(G,Y)>0,\qquad \rho_4(G,Z)>0.
\]
Thus $4$-bilightness rules out two separated regions that both have positive density behind boundaries of size at most four. Every 5-connected graph is $4$-bilight. Indeed, a set with boundary at most four cannot have both itself and the complement of its closed neighbourhood nonempty.

We prove the following stronger statement for $4$-bilight graphs.
\begin{theorem}\label{thm:bilight}
Let $G$ be a $4$-bilight graph with $n\ge4$ vertices and at least $4n-9$ edges. Then $G$ contains $\Ke$ as a minor or contains $K_6$ as a subgraph.
\end{theorem}
\begin{proof}[Proof of Theorem~\ref{thm:main} from Theorem~\ref{thm:bilight}]
Apply Theorem~\ref{thm:bilight}. Suppose that $G$ contains a 6-vertex clique $C$. Since $n\ge7$, there is a nonempty component $U$ of $G-C$. 5-connectivity gives $|N_G(U)\cap C|\ge5$. Contracting $U$ to a vertex and retaining $C$ gives a $K_7^{-}$ minor, which contains $\Ke$.
\end{proof}
We also prove the following auxiliary theorem.
\begin{theorem}\label{thm:aux}
Every $4$-bilight graph with $n\ge3$ vertices and at least $4n-8$ edges contains $\Ke$ as a minor.
\end{theorem}

The bound in Theorem~\ref{thm:main} is sharp.
For each $\ell\ge4$, let $G_\ell=P_3\vee C_\ell$, where $\vee$ denotes the join of two disjoint graphs. Then $n=\ell+3$ and
\[
e(G_\ell)=2+\ell+3\ell=4n-10.
\]
Deleting at most four vertices leaves a connected graph. If both joined parts survive, this is immediate. If all three vertices of $P_3$ are deleted, at most one vertex of the cycle is deleted. Deleting the entire cycle is possible only when $\ell=4$, in which case $P_3$ survives. Thus $G_\ell$ is 5-connected. Deleting the middle vertex of $P_3$ leaves the planar bipyramid $2K_1\vee C_\ell$. An apex graph has no $K_6$ minor. A model avoiding the apex would lie in a planar graph, and deleting the branch set containing the apex from any other $K_6$ model would leave a $K_5$ model in a planar graph. Since $\Ke$ contains a $K_6$ minor (contract an edge joining endpoints of the two missing edges), $G_\ell$ has no $\Ke$ minor.

The exception in Theorem~\ref{thm:bilight} is also necessary. For $t\ge1$, consider
\begin{equation}\label{eq:Ft}
F_t=K_4\vee(K_2\sqcup tK_1).
\end{equation}
This graph has $n=t+6$ vertices and $4n-9$ edges. Two nonempty nonadjacent sets avoid the universal $K_4$. A set outside this $K_4$ with boundary at most four has positive $4$-density precisely when it contains both vertices of the distinguished $K_2$. Hence two such positive sets cannot be disjoint, and $F_t$ is $4$-bilight.

To see that $F_t$ has no $\Ke$ minor, suppose that a model exists. At most four of its seven branch sets meet the universal $K_4$, so at least three branch sets avoid it. Every induced subgraph of $\Ke$ on at least three vertices is connected. Those branch sets must therefore lie in a single component of $K_2\sqcup tK_1$, which has at most two vertices, a contradiction. Finally, $F_t$ contains a $K_6$ subgraph, as required by Theorem~\ref{thm:bilight}.

Our proof follows the reduction method of Dvo\v r\'ak, Norin and Rahman~\cite{DNR26}. Sections~\ref{sec:preliminaries} and~\ref{sec:irreducible} adapt their reduction, linkage and separation arguments, keeping track of the number $b\in\{3,4\}$ of vertices that must remain after a reduction. The rooted-minor results of~\cite{Dvo26,DNR26,NT25}, collected in Theorem~\ref{thm:inputs}, are used with their original hypotheses.

Section~\ref{sec:rooted} proves the rooted-minor lemmas needed at the lower density. Lemma~\ref{lem:four} extends the internally 4-connected result to $4$-light 4-rooted graphs of positive density. Theorem~\ref{thm:five} gives a model in $\W^+$ at rooted density two, with at most one missing edge incident with a nonroot. This allows one root edge to remain missing on the other side of a 5-separation, provided the two missing edges are independent. Lemma~\ref{lem:joining} uses this choice to join the sides when $r_1+r_2\ge m+1$, where $m$ counts the missing separator edges.

Section~\ref{sec:proofs} first proves Theorem~\ref{thm:aux}. At density $4n-9$, this theorem forces every admissible reduction to preserve density exactly. The density increment is the sum of the negative density of the removed fragment and the number of added boundary edges. Both are nonnegative, so both must vanish. Thus no edge is added between surviving vertices, and no new $K_6$ subgraph is created. This permits the minimal-counterexample proof of Theorem~\ref{thm:bilight}.

\section{Preliminaries}\label{sec:preliminaries}
Fabila-Monroy and Wood~\cite{FMW13} characterized the graphs that contain a $K_4$ minor rooted at four prescribed vertices. Here we use density conditions to force rooted targets, principally with four or five roots.

A \emph{rooted graph} is a graph $R$ together with a specified root set $X_R$. A rooted model consists of pairwise disjoint nonempty connected branch sets, one for each target vertex, such that target edges are represented by edges between the corresponding branch sets. Each root branch set contains its prescribed root, and no other original root. Extra edges may be deleted. When the target has an independent root set, edges between original roots neither help nor obstruct such a model. We write
\[
\rho_4(R)=e(R)-e(R[X_R])-4|V(R)\setminus X_R|.
\]
For an unrooted graph we instead write $\rho_4(G)=e(G)-4|V(G)|$. The meaning will always be determined by whether roots have been specified.

A \emph{fragment} of a rooted graph is a nonempty set avoiding the roots. In an unrooted graph any nonempty vertex set is a fragment. A fragment $Y$ is a \emph{$k$-fragment} if $|\dd Y|=k$, and a \emph{$(\le k)$-fragment} if $|\dd Y|\le k$. Its associated rooted graph is
\[
R_Y=(G[Y\cup\dd Y],\dd Y),\qquad \rho_4(R_Y)=\rho_4(G,Y).
\]
A $k$-rooted graph is \emph{$4$-light} if every $(\le k-1)$-fragment has nonpositive $4$-density. It is \emph{internally $k$-connected} if there is no such fragment. A pair of disjoint nonadjacent fragments is a \emph{bifragment}. It is \emph{dense} if both fragment densities are positive.

A \emph{separation} $(A,B)$ of $G$ has $A\cup B=V(G)$ and no edge between $A\setminus B$ and $B\setminus A$. Its order is $|A\cap B|$, and its right-hand side is $R_{A,B}=(G[B],A\cap B)$. In a rooted graph, a \emph{root separation} additionally satisfies $X_G\subseteq A$. A root separation of order $|X_G|$ is \emph{proper} if $A\ne X_G$. Then $|B|<|V(G)|$. Right-hand sides of order at most the number of original roots inherit $4$-lightness whenever they are contained in a $4$-light rooted graph. Indeed, a smaller-boundary fragment in the right-hand side has exactly the same boundary and density in the original graph.

For a root separation $(A,B)$, an \emph{isolator} is a root separation $(C,D)$ of minimum order subject to $C\subseteq A$ and $B\subseteq D$. By Menger's theorem, its order equals the largest number of vertex-disjoint paths from $X_G$ to $B$. We choose these paths to meet $X_G$ only at their initial vertices and $B$ only at their final vertices. A root in $X_G\cap B$ is represented by a zero-length path. The final vertices, called \emph{terminators}, belong to $A\cap B$.

If the maximum number of paths is $|A\cap B|=|X_G|$, any rooted model in $R_{A,B}$ transfers to the original roots by adding the paths to its root branch sets. More generally, restricting the paths to an isolator transfers a model on the terminators to its separator. No path internal vertex then meets a branch set in the interior of $B$.

A \emph{dart} is a 4-rooted graph whose root graph is $P_3\sqcup K_1$ and whose unique nonroot vertex is adjacent to all four roots. In particular, its rooted $4$-density is zero. Let $D^*$ be the 4-rooted graph with independent roots and two adjacent nonroots, both adjacent to every root.

Let $\W$ consist of the 5-rooted graphs with independent roots and two nonroots $p,q$ satisfying the following condition. Either $p,q$ are nonadjacent and each is adjacent to every root, or $pq$ is an edge and at most two of the ten root--nonroot edges are missing, with any two missing edges independent. These are the $K_{2,\downarrow5}$ and vampire templates of~\cite{DNR26}. Let $\W^+\subseteq\W$ consist of the templates missing at most one of the eleven possible edges with a nonroot end.

Figure~\ref{fig:templates} shows the 4-root target and representative weak and strong 5-root targets.
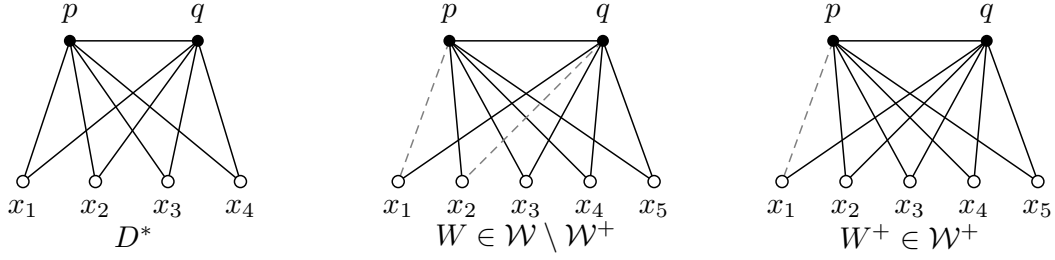
\begin{figure}[htbp]
\centering
\resizebox{\textwidth}{!}{%
\begin{tikzpicture}[x=1cm,y=1cm,line width=0.5pt,font=\small]
\path[use as bounding box] (-0.3,-0.9) rectangle (13.0,2.4);
\begin{scope}
\coordinate (p) at (0.65,1.65); \coordinate (q) at (2.15,1.65);
\foreach \i in {1,2,3,4} {\coordinate (x\i) at ({0.1+0.85*(\i-1)},0);}
\foreach \i in {1,2,3,4} {\draw (p)--(x\i); \draw (q)--(x\i);}
\draw (p)--(q);
\foreach \i in {1,2,3,4} {\node[draw,circle,fill=white,inner sep=0pt,minimum size=4pt,label=below:$x_\i$] at (x\i) {};}
\node[circle,fill=black,inner sep=0pt,minimum size=4pt,label=above:$p$] at (p) {};
\node[circle,fill=black,inner sep=0pt,minimum size=4pt,label=above:$q$] at (q) {};
\node at (1.4,-0.65) {$D^*$};
\end{scope}
\begin{scope}[xshift=4.5cm]
\coordinate (p) at (0.6,1.65); \coordinate (q) at (2.4,1.65);
\foreach \i in {1,2,3,4,5} {\coordinate (x\i) at ({0.75*(\i-1)},0);}
\foreach \i in {2,3,4,5} {\draw (p)--(x\i);}
\foreach \i in {1,3,4,5} {\draw (q)--(x\i);}
\draw (p)--(q); \draw[densely dashed,gray] (p)--(x1) (q)--(x2);
\foreach \i in {1,2,3,4,5} {\node[draw,circle,fill=white,inner sep=0pt,minimum size=4pt,label=below:$x_\i$] at (x\i) {};}
\node[circle,fill=black,inner sep=0pt,minimum size=4pt,label=above:$p$] at (p) {};
\node[circle,fill=black,inner sep=0pt,minimum size=4pt,label=above:$q$] at (q) {};
\node at (1.5,-0.65) {$W\in\mathcal W\setminus\mathcal W^+$};
\end{scope}
\begin{scope}[xshift=9cm]
\coordinate (p) at (0.6,1.65); \coordinate (q) at (2.4,1.65);
\foreach \i in {1,2,3,4,5} {\coordinate (x\i) at ({0.75*(\i-1)},0);}
\foreach \i in {2,3,4,5} {\draw (p)--(x\i);}
\foreach \i in {1,2,3,4,5} {\draw (q)--(x\i);}
\draw (p)--(q); \draw[densely dashed,gray] (p)--(x1);
\foreach \i in {1,2,3,4,5} {\node[draw,circle,fill=white,inner sep=0pt,minimum size=4pt,label=below:$x_\i$] at (x\i) {};}
\node[circle,fill=black,inner sep=0pt,minimum size=4pt,label=above:$p$] at (p) {};
\node[circle,fill=black,inner sep=0pt,minimum size=4pt,label=above:$q$] at (q) {};
\node at (1.5,-0.65) {$W^+\in\mathcal W^+$};
\end{scope}
\end{tikzpicture}%
}
\caption{Rooted templates. Open vertices are roots and filled vertices are nonroots. Dashed segments are missing edges, not edges of the graph. The middle template has two independent root--nonroot defects, whereas the right template has only one. The figures show representatives, not all members of $\W$ and $\W^+$. Diagram code prepared with OpenAI Codex (GPT-6).}
\label{fig:templates}
\end{figure}

A 5-rooted graph is \emph{quite heavy} if its density is at least two, or if its density is one and it has no nonroot adjacent to all five roots.

For a graph on five vertices, let $\Sc_{5,t}$ be the class of graphs with at most $t$ edges, and let $\Sc^-_{5,4}$ exclude $K_3\sqcup K_2$ from $\Sc_{5,4}$. Define
\[
\begin{array}{c|ccccccccc}
r&r\le0&1&2&3&4&5&6&r\ge7\\
\T_r&\Sc_{5,0}&\Sc_{5,1}&\Sc_{5,3}&\Sc^-_{5,4}&\Sc_{5,6}&\Sc_{5,8}&\Sc_{5,9}&\Sc_{5,10}
\end{array}.
\]
A 5-rooted graph is $\T_r$-\emph{universal} if it has every member of $\T_r$ as a rooted minor for every prescribed bijection onto its roots.
\needroom{15\baselineskip}
\begin{theorem}\label{thm:inputs}
The following statements hold.
\begin{enumerate}
\item Every $4$-light 5-rooted graph $R$ is $\T_{\rho_4(R)}$-universal~\cite[Theorem~4]{Dvo26}. See also~\cite[Theorem~2.6]{DNR26}.
\item If a $4$-light $k$-rooted graph has positive density and $k\le3$, then it has a rooted $K_k$ minor. If $k=4$, it has a rooted dart minor~\cite[Theorem~2.7 and Lemma~3.1]{DNR26}.
\item Every quite heavy $4$-light 5-rooted graph has a rooted minor in $\W$~\cite[Theorem~2.9]{DNR26}.
\item If $(R,X)$ is internally 4-connected, $|X|=4$, and $e(R)\ge4|V(R)|-9$, then it has an $X$-rooted $D^*$ minor~\cite[Lemma~12]{NT25}.
\end{enumerate}
\end{theorem}
We also use the elementary 3-root classification~\cite[Lemma~11]{Dvo26}. A 3-rooted graph with no rooted triangle has a skeleton with at most one nonroot, and every component outside the skeleton has at most two neighbours in it. Here a skeleton is a supergraph of the induced graph on its vertex set, containing all roots. Only the stated bound on its size and on attachment sets is used below.

A $(\le4)$-fragment $Y$ is \emph{reducible} if $\rho_4(G,Y)\le0$ and either $R_Y$ has a rooted clique on its boundary, or $|\dd Y|=4$, $|Y|\ge2$, and $R_Y$ has a dart minor. In the unrooted case we additionally require $|V(G)\setminus Y|\ge3$. Its \emph{reducent} is
\[
(G-Y)\cup D,
\]
where $D$ is the rooted boundary clique or the dart supplied by the model. In the latter case its nonroot vertex is new. The reducent is a minor and has fewer vertices. It preserves original roots in the rooted case. If $a$ boundary edges are added, then in an unrooted graph
\begin{equation}\label{eq:increment}
\rho_4((G-Y)\cup D)-\rho_4(G)=-\rho_4(G,Y)+a\ge0.
\end{equation}
Rooted density is also nondecreasing. Edges added between original roots are simply omitted from that count.
\needroom{6\baselineskip}
\begin{lemma}\label{lem:stability}
\begin{enumerate}
\item In a $4$-light 5-rooted graph, reducing an inclusionwise maximal reducible fragment preserves $4$-lightness.
\item Let $G$ be an unrooted $4$-bilight graph, and let $b\in\{3,4\}$. Choose $Y$ inclusionwise maximal among reducible fragments satisfying $|V(G)\setminus Y|\ge b$. Then its reducent is $4$-bilight.
\end{enumerate}
\end{lemma}
\begin{proof}
Assertion (1) is~\cite[Lemma~3.3]{DNR26}. We give the unrooted argument with the complement-size condition explicit. It also proves the case $b=3$ of (2).

Let $H=(G-Y)\cup D$ be the reducent, put $Q=\dd_GY$, and let $y$ denote the new nonroot vertex when $D$ is a dart. Suppose that $H$ is not $4$-bilight. By~\cite[Lemma~3.2]{DNR26}, choose a dense $(\le4)$-bifragment $(S,T)$ of $H$ such that either $V(D)$ avoids both interiors, or
\[
V(D)\cap S\ne\varnothing,\quad V(D)\cap T=\varnothing,\quad
|V(D)\setminus(S\cup\dd_HS)|\le1,
\]
and the rooted graph $R^H_S$ is $4$-light. In the first case neither the added boundary edges nor the new dart vertex contributes to the densities of $S,T$. Also neither interior has a neighbour in $Y$, since both avoid $Q$. Thus the same dense small bifragment exists in $G$, a contradiction.

In the second case, Theorem~\ref{thm:inputs}(2) gives a boundary clique model in $R^H_S$ when $|\dd_HS|\le3$, and a dart model when $|\dd_HS|=4$. We use the normalization from the proof of~\cite[Lemma~3.3]{DNR26}, as follows. If $D$ is a dart and its unique vertex outside $S\cup\dd_HS$ is $z$, then $y\in\dd_HS$. In this case, add $y$ to $S$ and replace the boundary root $y$ by $z$, using the edge $yz$ in its root branch set. If $V(D)\subseteq S\cup\dd_HS$ and $y\in\dd_HS$, add $y$ to $S$ and cancel that boundary root. A clique restricts to the remaining roots. A dart, with one root cancelled, has a triangle model on its other three roots. In all other cases leave $S$ unchanged.

Call the resulting fragment $S_1$. This construction keeps $S_1$ nonadjacent to $T$, gives $|\dd_HS_1|\le4$, and ensures
\[
S\subseteq S_1,\qquad V(D)\subseteq S_1\cup\dd_HS_1,\qquad
V(D)\setminus Q\subseteq S_1.
\]
Its associated rooted graph has a boundary clique or dart model. For completeness, the assertion about cancelling a dart root follows directly from its root path $abc$ and isolated root $d$. If the cancelled root is $a$ or $c$, absorb the nonroot into $d$. If it is $d$, absorb the nonroot into $a$. If it is $b$, absorb $b$ into $c$ and the nonroot into $d$. In each case the three remaining roots form a triangle model.

Undo the reduction and set
\[
Y'=Y\cup\bigl(S_1\setminus(V(D)\setminus Q)\bigr).
\]
Then $\dd_GY'=\dd_HS_1$. Composing with the model defining the reduction lifts the boundary clique or dart to $R^G_{Y'}$. Since $S$ is positive and has boundary at most four, $|S|\ge2$. At most the single new dart vertex is removed from $S_1$, so $|Y'|\ge|Y|+1\ge2$.

The fragment $T$ and its density are unchanged in $G$, and $Y'$ is nonadjacent to $T$. Hence $4$-bilightness of $G$ forces $\rho_4(G,Y')\le0$. Furthermore,
\[
T\cup\dd_GT\subseteq V(G)\setminus Y'.
\]
A positive fragment $T$ with $k=|\dd_GT|\le4$ has $|T|+k\ge6$. Indeed, if $t=|T|$ and $t+k\le5$, then
\[
\rho_4(G,T)\le\binom{t}{2}+tk-4t\le0.
\]
Consequently $|V(G)\setminus Y'|\ge6\ge b$. Thus $Y'$ is an admissible reducible fragment strictly containing $Y$, contrary to its maximality. This proves (2).
\end{proof}
\needroom{10\baselineskip}
We will need the following consequence of the linkage argument in~\cite{DNR26}. It applies even when the whole graph already has a weak model.
\begin{lemma}\label{lem:nonfull}
Let $R$ be a $4$-light rooted graph with $k\le5$ roots, with $\rho_4(R)\le0$ if $k\le4$. Suppose that a root separation $(A,B)$ of order five has a right-hand side containing a member of $\W$. If there is no linkage of size five from the original roots to $B$, then $R$ has a reducible fragment containing $B\setminus A$.
\end{lemma}
\begin{proof}
This is the non-full-linkage branch of~\cite[Corollary~4.2]{DNR26}. We recall the argument to specify the conclusion used here. Take an isolator $(C,D)$ of order $q\le4$. If $q=k$, choose $(C,D)=(X_R,V(R))$. Otherwise $q<k$. The rooted graph $R_{C,D}$ has nonpositive density, by the hypothesis when $q=k$ and by $4$-lightness otherwise. Restrict the weak template to the $q$ terminators of a linkage within $D$. It supplies a rooted $K_q$, or a dart when $q=4$~\cite[Observation~4.1]{DNR26}. Transfer that model to $C\cap D$. The fragment $D\setminus C$ has at least two vertices, because it contains the two nonroot branch sets of the template in $B\setminus A$. It is therefore reducible. If its actual boundary is smaller than the separator, a clique model restricts to that boundary. A dart model forces every one of its four roots to have a neighbour in the interior.
\end{proof}
\needroom{6\baselineskip}
\begin{lemma}\label{lem:cancel}
Let $R$ be a $4$-light graph with at most four roots and $\rho_4(R)\le0$. Changing some of its roots into nonroots preserves both properties.
\end{lemma}
\begin{proof}
This is~\cite[Observation~5.2]{DNR26}. Let $X'$ be the remaining roots, put $M=X_R\setminus X'$, and let $H$ denote the graph with root set $X'$. Consider a fragment $Y$ of $H$ with either $Y=V(H)\setminus X'$ or $|\dd_HY|<|X'|$. Set $Y'=Y\setminus M$. In the first case, $Y'=V(R)\setminus X_R$. In the second,
\[
|\dd_RY'|\le |\dd_HY|+|M|<|X_R|.
\]
Hence $\rho_4(R,Y')\le0$, with the density of the empty set taken to be zero. Put $a=|M\cap Y|$ and $t=|\dd_HY|$. In either case $|M|+t\le|X_R|\le4$. The edges counted for $Y$ but not for $Y'$ number at most $\binom{a}{2}+at\le4a$. Therefore
\[
\rho_4(H,Y)\le\rho_4(R,Y')+\binom{a}{2}+at-4a\le0.
\]
This proves both assertions.
\end{proof}
We also use the density bound of~\cite[Theorem~2.8]{DNR26}:
\begin{equation}\label{eq:known}
|V(G)|\ge3,\quad G\text{ is $4$-bilight},\quad e(G)\ge4|V(G)|-7\quad\Longrightarrow\quad \Ke\preccurlyeq G.
\end{equation}

\needroom{12\baselineskip}
\section{Rooted minors}\label{sec:rooted}
\subsection{Four roots}\label{subsec:four}
\begin{lemma}\label{lem:four}
Every $4$-light 4-rooted graph of positive $4$-density has a rooted $D^*$ minor.
\end{lemma}
\begin{proof}
Choose a counterexample $R$ with the fewest vertices and delete edges between its four roots. In this proof, call a nonempty fragment \emph{clique-reducible} if its boundary has size at most three and its associated rooted graph has a boundary clique minor. Its density is automatically nonpositive.

There is no clique-reducible fragment. Otherwise choose an inclusionwise maximal one, $Y$, and replace it by the clique on $Q=\dd_RY$, obtaining a smaller rooted minor $R'$ of at least the same density. If $R'$ is not $4$-light, choose an inclusionwise minimal positive $(\le3)$-fragment $Z$. Its associated rooted graph is $4$-light and has a boundary clique minor by Theorem~\ref{thm:inputs}(2). If $Z\cap Q=\varnothing$, no new clique edge contributes to its density, and $Z$ is the same forbidden positive small fragment in $R$. Otherwise the clique property gives $Q\subseteq Z\cup\dd_{R'}Z$. Undoing the reduction gives $Y\cup Z$ with boundary contained in $\dd_{R'}Z$. Compose the two clique models, and restrict to the actual boundary if it is smaller. Its density is nonpositive by $4$-lightness, so it is a clique-reducible fragment strictly containing $Y$. This contradiction shows that $R'$ is $4$-light. Minimality now gives a $D^*$ model in $R'$, and hence in $R$, again a contradiction.

There is also no proper root 4-separation with positive right-hand side. That side is a smaller $4$-light 4-rooted graph, so it has a $D^*$ model. A linkage of size four transfers the model. Otherwise an isolator has order $q\le3$. The restriction of $D^*$ to any $q$ of its roots has a rooted $K_q$ minor. For $q=3$, absorb its two nonroots into two distinct roots. For $q=2$, absorb one nonroot into a root. The case $q\le1$ is immediate. Transfer this clique to the isolator separator. Its interior is nonempty, has nonpositive density by $4$-lightness, and is clique-reducible, a contradiction.

Every nonroot has degree at least five. Deleting a nonroot $v$ of degree at most four does not decrease density. If the resulting graph is $4$-light, minimality is contradicted. Otherwise let $Z$ be a positive small fragment after deletion. It must have been adjacent to $v$. Restoring $v$ to its boundary produces a positive fragment of boundary exactly four. The left side of this root separation contains the four original roots and $v$. It is a proper root 4-separation, which was just excluded.

We claim that $R$ is internally 4-connected. A fragment with boundary zero or one is clique-reducible. If a fragment has boundary two, each component of its interior meets both boundary vertices, since otherwise there would be a fragment of smaller boundary. It therefore supplies a path between the two boundary vertices and is clique-reducible. For a 3-boundary fragment, a rooted triangle also makes it clique-reducible. If there is no rooted triangle, apply the 3-root skeleton classification stated after Theorem~\ref{thm:inputs}. Every component outside the skeleton is a fragment of the entire graph with boundary at most two, and so none exists. The original fragment then has at most one vertex. Such a vertex has degree at most three, contrary to the preceding paragraph.

Complete the four roots to a clique, obtaining $R^*$. This preserves internal 4-connectivity, and
\[
e(R^*)=4(|V(R)|-4)+\rho_4(R)+6\ge4|V(R)|-9.
\]
Theorem~\ref{thm:inputs}(4) gives a rooted $D^*$ model. Each added edge joins distinct original roots, which belong to distinct root branch sets. Thus no added edge is used within a branch set or for an edge of the independent-root target. The model already exists in $R$, a contradiction.
\end{proof}
\needroom{6\baselineskip}
\begin{lemma}\label{lem:normal}
Suppose a 4-rooted graph $R$ has a rooted $D^*$ minor. Its two nonroot branch sets can be chosen as $A,B$ so that
\[
\dd_R(A\cup B)=\{d_1,d_2,d_3,d_4\},
\]
every $d_i$ has a neighbour in both $A$ and $B$, and the four roots have pairwise disjoint paths to the corresponding $d_i$ avoiding $A\cup B$.
\end{lemma}
\begin{proof}
Maximize $|A\cup B|$, and then choose inclusionwise minimal trees for the root branch sets. Consider a leaf other than the prescribed root of one of these trees. If it has no neighbour in either $A$ or $B$, delete it. If it meets exactly one of $A,B$, move it to that branch set. The tree's parent edge preserves adjacency to that set, and adjacency to the other set is unaffected. Both alternatives contradict the choice of model. Thus every such leaf meets both $A$ and $B$.

If another vertex in the same root tree meets $A$, move the leaf to $B$. The other vertex preserves adjacency to $A$, and the parent edge preserves adjacency to $B$. This increases $|A\cup B|$, a contradiction. The same reasoning applies with $A,B$ interchanged. Hence a nontrivial root tree has a unique nonroot leaf, its only vertex meeting either nonroot branch set. The tree is a path to this leaf. A singleton root tree already has the required form. Call these four endpoints $d_i$. Any unused vertex adjacent to $A\cup B$ could be absorbed into a nonroot branch set. Maximality rules this out, proving the assertion about the entire external neighbourhood.
\end{proof}

\needroom{14\baselineskip}
\subsection{Five roots}\label{subsec:five}
\begin{theorem}\label{thm:five}
Every $4$-light 5-rooted graph $R$ with $\rho_4(R)\ge2$ has a rooted minor in $\W^+$.
\end{theorem}
The proof contracts a suitable edge in a minimal counterexample and applies Lemmas~\ref{lem:four} and~\ref{lem:normal} to a positive fragment of the contracted graph. We first need a local lemma about rooted stars.
\needroom{12\baselineskip}
\begin{lemma}\label{lem:star}
Let $L$ be a graph, let $X\subseteq V(L)$, and suppose that
\[
5\le|V(L)|,\qquad |V(L)|+|X|\le8.
\]
Suppose that vertices outside $X$ have degree at least four and vertices in $X$ have degree at least $5-|X|$. For every 5-set $Z\supseteq X$, there is a $Z$-rooted $K_{1,4}$ minor whose centre is a vertex of $Z\setminus X$.
\end{lemma}
\begin{proof}
Put $U=V(L)\setminus Z$. Then $|U|\le3$ and $|X|\le3-|U|$. If $U=\varnothing$, any vertex in $Z\setminus X$ is already a centre. Suppose that $L[U]$ has an isolated vertex $u$. It is adjacent to at least four vertices of $Z$. If it meets all five, absorb it into any $z\in Z\setminus X$. Otherwise let $y$ be its unique nonneighbour in $Z$. If $y\notin X$, then $y$ has at least $5-|U|$ neighbours in $Z$, which is more than $|X|$. If $y\in X$, it has at least $6-|X|-|U|$ neighbours in $Z$, which is more than $|X|-1$ since $|X|+|U|\le3$. Thus some neighbour $z$ of $y$ lies in $Z\setminus X$. As $y$ is the only root missed by $u$, the edge $uz$ is present. The connected centre branch set $\{u,z\}$ meets the other four roots.

It remains that $U$ has two or three vertices and $L[U]$ has no isolated vertex. Hence it is connected. If a vertex of $Z\setminus X$ has no neighbour in $U$, its degree makes it adjacent to all other roots and it is a centre. Otherwise every vertex of $Z\setminus X$ meets $U$. When $|U|=3$, we have $X=\varnothing$, and absorbing $U$ into any root gives the star. When $|U|=2$, we have $|X|\le1$. If its possible vertex $x$ also meets $U$, the same construction works. If it does not, its degree is at least four, so it meets every other root. Absorb $U$ into any root outside $X$. Unused vertices may be deleted.
\end{proof}
Suppose Theorem~\ref{thm:five} is false, and choose a counterexample $G$ minimizing its number of vertices, and then its number of edges. Let $S$ be its five roots, and write $d_S(v)=|N_G(v)\cap S|$. Root edges can be deleted, so $G[S]$ is independent. We work with this fixed counterexample until the proof of Theorem~\ref{thm:five} is complete.
\needroom{6\baselineskip}
\begin{lemma}\label{lem:irreducibleroot}
$G$ has no reducible $(\le4)$-fragment and no proper root 5-separation whose right-hand side has density at least two.
\end{lemma}
\begin{proof}
An inclusionwise maximal reducible fragment could be reduced using Lemma~\ref{lem:stability}(1), producing a smaller $4$-light counterexample of at least the same density. For the second assertion, a smaller right-hand side inherits $4$-lightness and has a $\W^+$ model by minimality. A full linkage transfers this strong model. If no full linkage exists, Lemma~\ref{lem:nonfull} gives the forbidden reducible fragment. The full-linkage case must transfer the $\W^+$ model. The existence of a weak $\W$ model alone is consistent with the choice of $G$.
\end{proof}
\needroom{6\baselineskip}
\begin{lemma}\label{lem:rootdegrees}
The graph $G-S$ is connected, every nonroot has at most three neighbours in $S$, and every root has degree at least two. Moreover
\[
\rho_4(G)=2,\qquad |V(G)\setminus S|\ge5.
\]
\end{lemma}
\begin{proof}
Every positive-density component of $G-S$ meets all five roots. Two such components give a rooted $K_{2,\downarrow5}$ after contraction. There is therefore exactly one positive-density component, and its density is at least two. Lemma~\ref{lem:irreducibleroot} implies that it contains all nonroots. In particular, $G-S$ is connected and every root has a neighbour.

Suppose a nonroot $z$ has $d\ge4$ root neighbours. Let $Q_1,\ldots,Q_t$ be the components of $G-(S\cup\{z\})$. Each meets $z$, and
\begin{equation}\label{eq:components}
\sum_i\rho_4(G,Q_i)=\rho_4(G)+4-d.
\end{equation}
If $d=5$, some component has positive density and thus has at least four root neighbours. Contracting it gives a $\W^+$ model with $z$. Suppose $d=4$, and let $x$ be the missing root neighbour of $z$. A component meeting all five roots gives the same contradiction. Every positive-density $Q_i$ consequently has boundary $z$ and four roots. Its density cannot be at least two, by Lemma~\ref{lem:irreducibleroot}, since its root separation is proper. It has density one. Equation~\eqref{eq:components} now implies that there are at least two positive components. If a positive component meets $x$, merge it with $z$ and retain another positive component. Otherwise merge $z$ with any component meeting $x$, which exists, and retain a positive component. The merged branch set meets all five roots, the retained one meets at least four, and they are adjacent. This is a $\W^+$ model.

If a root $x$ has unique neighbour $v$, replacing $x$ by $v$ as a root on the remaining graph gives a proper root 5-separation whose right-hand side has density $\rho_4(G)+4-d_S(v)\ge3$. This contradicts Lemma~\ref{lem:irreducibleroot}.

If $\rho_4(G)\ge3$, delete any edge. Minimality forces the resulting graph to fail $4$-lightness. Let $Y$ be a positive small fragment there. The deleted edge must join $Y$ to a vertex $u$ outside its closed neighbourhood. Otherwise $Y$ already violates lightness in $G$. Restoring that edge gives a positive fragment of boundary exactly five and density at least two. Its root separation cannot be proper by Lemma~\ref{lem:irreducibleroot}. Its left side is therefore $S$, and $u\in S$ has only the restored edge as a neighbour, contrary to its degree. Thus $\rho_4(G)=2$.

Write $k=|V(G)\setminus S|$. If $k\le3$, the bound on the number of root neighbours of each nonroot gives $\rho_4(G)\le\binom{k}{2}-k\le0$. If $k=4$, equality with density two forces the nonroots to induce $K_4$, each with three root neighbours. Regard each missing root pair as an edge of a multigraph on $S$. It has four edges and maximum degree at most two, because every root has degree at least two in $G$. If an edge is repeated, its endpoints are saturated, and pair its two copies with the other two edges. Both pairs consist of disjoint edges. Otherwise this multigraph is $P_5$, $C_4\sqcup K_1$, or $C_3\sqcup K_2$. In each case its edges can be partitioned into two pairs, one disjoint and the other sharing at most one endpoint. Pair the four nonroots accordingly and contract each pair. One resulting vertex meets all roots and the other at least four. They are adjacent. Again this is a $\W^+$ model. Hence $k\ge5$.
\end{proof}
For an edge $uv$ with both ends nonroots, put $\tau(uv)=|N(u)\cap N(v)|$. For an edge $xv$ with $x\in S$, put
\[
\tau(xv)=|N(x)\cap N(v)|+d_S(v)-1.
\]
As $S$ is independent, this is the triangle count after completing $S$ to a clique. Counting the edges that merge or become root edges gives
\begin{equation}\label{eq:contraction}
\rho_4(G/e)=\rho_4(G)+3-\tau(e)=5-\tau(e).
\end{equation}
\needroom{6\baselineskip}
\begin{lemma}\label{lem:edge}
Some edge $e$ of $G$ satisfies $\tau(e)\le3$.
\end{lemma}
\begin{proof}
Suppose that $\tau(e)\ge4$ for every edge $e$.

Since
\[
\sum_{v\notin S}(d(v)+d_S(v)-8)=2\rho_4(G)=4
\]
and there are at least five nonroots, some nonroot $v$ satisfies $d(v)+d_S(v)\le8$. Put $L=G[N(v)]$ and $X=S\cap N(v)$. If $t\in N(v)\setminus S$, then
\[
d_L(t)=|N(v)\cap N(t)|=\tau(vt)\ge4.
\]
For $x\in X$, the definition of $\tau(vx)$ instead gives $d_L(x)\ge5-|X|$. The graph $G-S$ is connected and has at least five vertices, so $v$ has a nonroot neighbour. The preceding degree bound for that neighbour gives $|N(v)|\ge5$. Thus $|V(L)|+|X|\le8$, as required by Lemma~\ref{lem:star}.

Take an isolator $(C,D)$ of the root separation $(V(G)\setminus\{v\},N[v])$, and write $q=|C\cap D|\le5$. A maximum linkage from $S$ to $N[v]$ has $q$ paths and terminators $T\subseteq N(v)$. Use zero-length paths at the roots in $X$, so $X\subseteq T$. Extend $T$ to a 5-set $Z\subseteq N(v)$. The proof of Lemma~\ref{lem:star} provides a connected centre branch set $H_z$ containing $z\in Z\setminus X$, with the other four roots of $Z$ as singleton branch sets. The set $H_z$ meets $Z$ only at $z$ and is adjacent to every other vertex of $Z$. Since $S\cap N(v)=X\subseteq Z$ and $z\notin X$, it contains no original root. Also $v$ is adjacent to all these sets. Linkage interiors avoid $N[v]$, so they avoid the local model.

\smallskip\noindent\emph{Case 1. $q=5$.}
Here $T=Z$. Since $z\notin S$, its linkage path has positive length. Remove its last vertex $z$ from its root branch set and retain $H_z$ as a nonroot branch set. Retain $v$ as the other nonroot and use the other four linkage paths as root branch sets. The last edge of the path to $z$ joins its remaining, nonempty root branch set to $H_z$. Both nonroot branch sets meet the other four root branch sets, and they are adjacent to each other. Only the edge from $v$ to the root branch set formerly ending at $z$ may be absent. These seven sets are disjoint and preserve all original roots, so they give a $\W^+$ model, a contradiction.

\smallskip\noindent\emph{Case 2. $1\le q\le4$.}
Delete unused roots of $Z$, retaining $H_z$ if its centre is not in $T$. For $q=2$ or $3$, construct a clique as follows. If $z\in T$, absorb $v$ into any other terminator. That enlarged set and $H_z$ meet every remaining terminator. If $z\notin T$, absorb $H_z$ into one terminator and $v$ into a different one. For $q=3$ the third terminator meets both enlarged sets. For $q=2$ the two enlarged sets are adjacent. Each absorption uses an edge, so the branch sets remain connected, disjoint, and rooted at distinct vertices of $T$. For $q=1$, retain the single terminator. Connectivity of $G$ excludes $q=0$.

For $q=4$, if $z\in T$, keep $H_z$ as its root branch set. If $z\notin T$, absorb $H_z$ into a terminator $t_0\in T$. In either case the four root branch sets contain a 3-edge star and $v$ is a nonroot adjacent to all four. Retain exactly two edges of that star and delete any other root edges. The root graph is $P_3\sqcup K_1$, so the resulting model is a dart.

The $q$ disjoint linkage paths meet the $q$ vertices of $C\cap D$ one each. Their tails in $D$ transfer the local clique or dart to these separator roots. A separator vertex in $N[v]$ must itself be a terminator. Hence the transfer uses no other separator root within a branch set. Put $F=D\setminus C$. It avoids $S$, has boundary at most $q\le4$, and contains at least $|N[v]|-q\ge2$ vertices. Thus $\rho_4(G,F)\le0$ by $4$-lightness. A clique model restricts to the actual boundary if necessary, by discarding the other root branch sets. In the dart case each of the four separator roots has a neighbour in $F$, since its branch set is connected to the nonroot branch set in $F$ without using another root. Thus the actual boundary has size four. In either case $F$ is reducible, contrary to Lemma~\ref{lem:irreducibleroot}.
\end{proof}
\begin{proof}[Completion of the proof of Theorem~\ref{thm:five}]
Choose $uv$ as in Lemma~\ref{lem:edge}, and let $J=G/uv$, with contraction vertex $w$. An edge of $G$ has at most one root end, since $S$ is independent. If $uv$ has a root end, $w$ retains its identity. Otherwise $w$ is a nonroot. Thus $J$ still has five distinct roots. Equation~\eqref{eq:contraction} gives $\rho_4(J)\ge2$.

A rooted $\W^+$ model in $J$ would lift to $G$ by expanding $w$ within its branch set. Since $J$ is smaller than the chosen counterexample, it must therefore fail $4$-lightness. Choose an inclusionwise minimal positive $(\le4)$-fragment $Y$ in $J$, and put $Q=\dd_JY$. The associated rooted graph $R_Y^J=(J[Y\cup Q],Q)$ is $4$-light. Indeed, a positive fragment inside it with boundary smaller than $|Q|$ would also be a positive $(\le4)$-fragment properly contained in $Y$ in $J$.

If $w\notin Y\cup Q$, the boundary and density of $Y$ are unchanged in $G$, contradicting $4$-lightness. If $w\in Y$, it is a nonroot and both $u,v$ are nonroots. Theorem~\ref{thm:inputs}(2) supplies a boundary clique or dart model in $R_Y^J$. Expanding $w$ lifts this model to
\[
Y'=(Y\setminus\{w\})\cup\{u,v\}.
\]
Its boundary is $Q$, its density is nonpositive by $4$-lightness of $G$, and $|Y'|=|Y|+1\ge2$. Thus it is reducible, contrary to Lemma~\ref{lem:irreducibleroot}. We conclude that $w\in Q$.

The set $Y$ now survives unchanged when $w$ is expanded. Its boundary has at most five vertices, and its density cannot decrease. If its boundary had at most four vertices, it would violate $4$-lightness of $G$. Hence
\[
|Q|=4,\qquad \dd_GY=(Q\setminus\{w\})\cup\{u,v\}.
\]
In particular, both $u$ and $v$ have neighbours in $Y$. The separation
\[
\bigl(V(G)\setminus Y,\;Y\cup\dd_GY\bigr)
\]
is a proper root 5-separation because its left side contains every original root and also at least one nonroot among $u,v$. Let
\[
c=|N_G(u)\cap N_G(v)\cap Y|.
\]
Exactly $c$ pairs of edges from $u,v$ to $Y$ merge under contraction, so
\[
\rho_4(G,Y)=\rho_4(J,Y)+c.
\]
Its right-hand density cannot be at least two by Lemma~\ref{lem:irreducibleroot}. Since $\rho_4(J,Y)$ is a positive integer, this gives
\begin{equation}\label{eq:split}
\rho_4(G,Y)=\rho_4(J,Y)=1,\qquad c=0.
\end{equation}

At this point $R_Y^J$ has exactly four roots, namely $Q$. It is $4$-light by the minimal choice of $Y$, and its rooted density is $\rho_4(J,Y)=1>0$. Thus Lemma~\ref{lem:four} applies, and Lemma~\ref{lem:normal} gives its boundary normal form. Obtain adjacent nonempty connected sets $A,B\subseteq Y$, and put $U=A\cup B$. Their union has exactly four external neighbours in $R_Y^J$, each adjacent to both $A$ and $B$. This is also the entire boundary in $J$. A vertex outside $Y\cup Q$ cannot have a neighbour in $U\subseteq Y$.

If $w$ is not one of those four neighbours, expanding it changes none of these adjacencies. If $w$ is one of them but only one of $u,v$ meets $U$, that endpoint inherits all its adjacencies to $A,B$. In either case $U$ has boundary four in $G$ and a rooted $D^*$ model with singleton boundary roots. To obtain a dart from such a model, absorb one nonroot branch set into one boundary root, retain the other as the dart nonroot, and keep two edges of the resulting 3-edge root star. Because $|U|\ge2$ and $\rho_4(G,U)\le0$, this makes $U$ reducible, again impossible.

It follows that
\[
\dd_GU=\{u,v,d_1,d_2,d_3\}.
\]
Each of $A,B$ meets all three $d_i$ and at least one of $u,v$, while both $u$ and $v$ meet $U$. If either branch set meets both endpoints, $A,B$ and the five singleton boundary roots already form a $\W^+$ model. A linkage of size five transfers this model to $S$. If there is no such linkage, Lemma~\ref{lem:nonfull}, applied to the root separation with right side $G[U\cup\dd_GU]$, produces a forbidden reducible fragment. Thus this case is excluded.

In the remaining case, after interchanging $A,B$ and $u,v$ if necessary, $A$ meets $u$ but not $v$, and $B$ meets $v$ but not $u$. Choose the labels so that $v$ is an original nonroot. At least one endpoint of $uv$ is a nonroot. These branch sets form a weak $\W$ model on the boundary. Lemma~\ref{lem:nonfull} and irreducibility now force a linkage of size five from $S$ to $U\cup\dd_GU$.

Truncate each linkage path at its first vertex in that closed neighbourhood. Each path first enters the closed neighbourhood through its boundary, because a vertex outside it has no neighbour in $U$. There are five disjoint paths and exactly five boundary vertices, so their terminators are precisely those five vertices. Write $P_t$ for the path ending at $t\in\dd_GU$, and $s_t$ for its original root. The paths avoid $U$, meet the boundary only at their terminators, and meet $S$ only at their initial vertices. Since $v\notin S$, $P_v$ has positive length.

Define the five root branch sets by
\[
R_v=V(P_v)\setminus\{v\},\qquad
R_t=V(P_t)\quad(t\in\{u,d_1,d_2,d_3\}),
\]
and the two nonroot branch sets by $A$ and $B'=B\cup\{v\}$. Because $P_v$ has positive length, deleting its last vertex leaves a nonempty connected path $R_v$ containing $s_v$. No original root is moved into a nonroot branch set. The set $B'$ is connected because $v$ has a neighbour in $B$. All seven sets are pairwise disjoint. Indeed, the paths are disjoint, avoid $A\cup B$, and $v$ has been removed from its root branch set before being added to $B$. Each $R_t$ contains exactly its prescribed original root, and $A,B'$ contain none.

The construction is illustrated in Figure~\ref{fig:split}.
\begin{figure}[htbp]
\centering
\resizebox{\textwidth}{!}{%
\begin{tikzpicture}[x=1cm,y=1cm,line width=0.5pt,font=\small,every label/.style={fill=white,inner sep=1pt}]
\path[use as bounding box] (-0.7,-1.0) rectangle (13.0,3.35);
\foreach \i/\t in {0/u,1/v,2/d_1,3/d_2,4/d_3} {
 \coordinate (t\i) at ({0.95*\i},1.2);
 \coordinate (s\i) at ({0.95*\i},2.7);
 \draw[densely dashed] (s\i)--(t\i);
}
\coordinate (A) at (-0.25,0); \coordinate (B) at (4.05,0);
\draw (A)--(B);
\foreach \i in {0,2,3,4} {\draw (A)--(t\i);}
\foreach \i in {1,2,3,4} {\draw (B)--(t\i);}
\draw (t0)--(t1);
\draw[line width=1pt] (0.95,1.2)--(0.95,1.68);
\node[circle,fill=black,inner sep=0pt,minimum size=3pt] at (0.95,1.68) {};
\foreach \i/\t in {0/u,1/v,2/d_1,3/d_2,4/d_3} {
 \node[draw,circle,fill=white,inner sep=0pt,minimum size=4pt,label=below:$\t$] at (t\i) {};
 \node[draw,circle,fill=white,inner sep=0pt,minimum size=4pt,label=above:$s_{\t}$] at (s\i) {};
}
\node[draw,ellipse,fill=white,minimum width=0.55cm,minimum height=0.35cm] at (A) {$A$};
\node[draw,ellipse,fill=white,minimum width=0.55cm,minimum height=0.35cm] at (B) {$B$};
\node at (1.9,-0.7) {(a) The linkage and local sets};
\draw[->,line width=0.7pt] (4.9,1.5)--(5.9,1.5);
\begin{scope}[xshift=7.2cm]
\foreach \i in {0,1,2,3,4} {\coordinate (r\i) at ({0.95*\i},2.7);}
\coordinate (AA) at (0.9,0.7); \coordinate (BB) at (3.0,0.7);
\draw (AA)--(BB);
\foreach \i in {0,2,3,4} {\draw (AA)--(r\i);}
\foreach \i in {0,1,2,3,4} {\draw (BB)--(r\i);}
\draw[densely dashed,gray] (AA)--(r1);
\foreach \i/\t in {0/u,1/v,2/d_1,3/d_2,4/d_3} {
 \node[draw,circle,fill=white,inner sep=0pt,minimum size=4pt,label=above:$R_{\t}$] at (r\i) {};
}
\node[draw,ellipse,fill=white,minimum width=0.55cm,minimum height=0.35cm] at (AA) {$A$};
\node[draw,ellipse,fill=white,minimum width=0.7cm,minimum height=0.35cm] at (BB) {$B'$};
\node at (1.9,-0.7) {(b) The seven branch sets};
\end{scope}
\end{tikzpicture}%
}
\caption{The split case in Theorem~\ref{thm:five}. In (a), dashed vertical segments represent disjoint linkage paths. The marked last edge of $P_v$ joins $R_v$ to the vertex moved into $B'=B\cup\{v\}$. The paths may have different lengths, and paths other than $P_v$ may have length zero. In (b), the dashed segment is the sole adjacency that may be absent. Every root set still contains its original root $s_t$. Diagram code prepared with OpenAI Codex (GPT-6).}
\label{fig:split}
\end{figure}
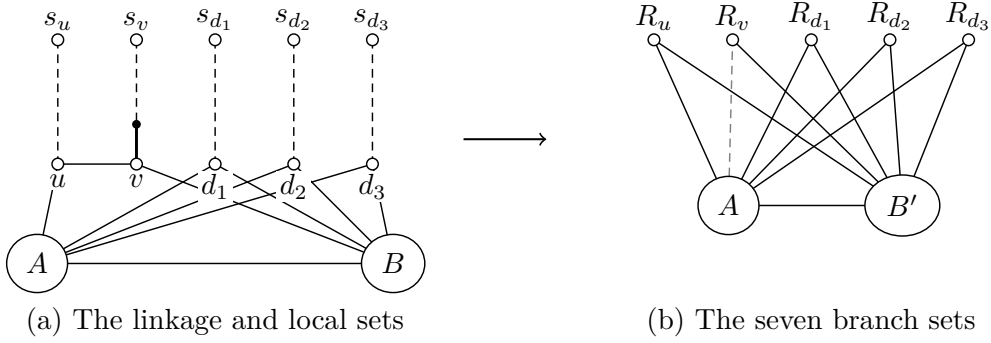

\needroom{12\baselineskip}
The following table gives witnesses for the required adjacencies.
\begin{center}
\begin{tabular}{ll}
\toprule
Pair of branch sets&Witness\\
\midrule
$A,B'$&an edge between $A$ and $B$\\
$A,R_u$&an edge from $A$ to $u$\\
$A,R_{d_i}$&an edge from $A$ to $d_i$, for $i=1,2,3$\\
$B',R_u$&the edge $vu$\\
$B',R_v$&the last edge of $P_v$\\
$B',R_{d_i}$&an edge from $B$ to $d_i$, for $i=1,2,3$\\
\bottomrule
\end{tabular}
\end{center}
The table accounts for the edge between the two nonroot branch sets, four root adjacencies of $A$, and all five root adjacencies of $B'$. These are ten distinct required adjacencies out of eleven. Only the edge between $A$ and $R_v$ may be absent. After deleting unwanted root edges, the seven sets give a rooted $\W^+$ model in $G$. This contradicts the choice of $G$ and proves the theorem.
\end{proof}
\begin{remark}
The density threshold two is best possible for Theorem~\ref{thm:five}. A vampire on seven vertices with exactly two independent missing root--nonroot edges is internally 5-connected as a rooted graph and has density one. It cannot have a $\W^+$ minor. Indeed, such a minor has the same seven vertices and would require at least ten nonroot-incident edges, whereas this vampire has nine.
\end{remark}

\needroom{10\baselineskip}
\subsection{Separations of order five}\label{subsec:joining}
\begin{lemma}\label{lem:joining}
Let $(A,B)$ be a 5-separation of a graph, and let $m$ be the number of missing edges on $S=A\cap B$. Suppose that both rooted sides are $4$-light, have densities $r_1,r_2\ge2$, and satisfy $r_1+r_2\ge m+1$. Then the graph contains $\Ke$ as a minor.
\end{lemma}
\begin{proof}
Use Theorem~\ref{thm:five} on the side of smaller density. The other side has density at least $s=\max\{2,\lceil(m+1)/2\rceil\}$. Write $F$ for the graph of missing edges on the five separator roots. The classes $\T_r$ are nested, so Theorem~\ref{thm:inputs}(1) supplies any target in $\T_s$. The following table records the relevant targets.
\begin{center}
\begin{tabular}{cccl}
\toprule
$m$&$s$&$\T_s$&target to be supplied\\
\midrule
$0$--$3$&2&$\Sc_{5,3}$&$F$\\
4&3&$\Sc^-_{5,4}$&$F-f$, or $F$ if it is a star\\
5&3&$\Sc^-_{5,4}$&$F-f\not\cong K_3\sqcup K_2$\\
6&4&$\Sc_{5,6}$&$F$\\
7&4&$\Sc_{5,6}$&$F-f$\\
8&5&$\Sc_{5,8}$&$F$\\
9&5&$\Sc_{5,8}$&$F-f$\\
10&6&$\Sc_{5,9}$&$F-f$\\
\bottomrule
\end{tabular}
\end{center}
Thus at most one root edge $f$ need remain absent. If the $\W^+$ model has no defect, any such $f$ is harmless. If its sole defect is the edge between its two nonroots, then $f$ is automatically independent of that defect. It remains to consider a missing root--nonroot edge $px$, where $x$ is a separator root. We choose $f$ to avoid $x$.

For $m=4$, if an edge $f$ of $F$ avoids $x$, then $F-f$ has three edges and belongs to $\Sc^-_{5,4}$. If no such edge exists, all four edges are incident with $x$, so $F=K_{1,4}$ and the whole of $F$ is an allowed target. For $m=5$, some edge avoids $x$ because $d_F(x)\le4$. If exactly one does, deleting it leaves a 4-edge star. If at least two do, at most one choice $f$ can leave $K_3\sqcup K_2$. In that event $F$ is this disjoint union plus an edge joining its components, and only that joining edge can be deleted to recover the disjoint union. Choose another edge avoiding $x$. In all cases the chosen target lies in $\Sc^-_{5,4}$.

For $m=0,1,2,3,6,8$, the target class contains $F$ itself, so no root edge is left missing. For $m=7$, deleting any edge leaves a 6-edge target in $\Sc_{5,6}$. For $m=9$ and $10$, deletion leaves respectively an 8-edge target in $\Sc_{5,8}$ and a 9-edge target in $\Sc_{5,9}$. In these last three rows, $m>4\ge d_F(x)$ guarantees an edge $f$ avoiding $x$.

Finally compose the two rooted models, identifying corresponding root branch sets through their common separator vertex. Each union is connected. Different root sets remain disjoint, and the two nonroot branch sets lie only on the first side. Retain all existing root edges. The seven resulting branch sets miss at most two independent edges, and hence give a $\Ke$ model.
\end{proof}

\section{Irreducible graphs}\label{sec:irreducible}
Fix $b\in\{3,4\}$. Call an unrooted graph \emph{$b$-irreducible} if it has no reducible $(\le4)$-fragment $Y$ with $|V(G)\setminus Y|\ge b$. Throughout this section, $G$ is a $4$-bilight, $\Ke$-minor-free, $b$-irreducible graph. We adapt the separation and contraction arguments of~\cite{DNR26}, checking that each reduction leaves at least $b$ original vertices. The corresponding results are cited in the individual proofs.
\needroom{6\baselineskip}
\begin{lemma}\label{lem:enclosing}
Let $(C,D)$ be a separation of order at most four such that $R_{C,D}$ is $4$-light and has nonpositive density. Suppose that
\[
|(C\setminus D)\cup\dd_G(C\setminus D)|\ge b.
\]
There is no 5-separation $(A,B)$ with $B\subseteq D$ for which $R_{A,B}$ is both $4$-light and quite heavy.
\end{lemma}
\begin{proof}
This is the enclosing argument of~\cite[Corollary~5.3]{DNR26}, with $b$-irreducibility in place of minimal-counterexample irreducibility. Its explicit closed-interior hypothesis ensures that the forbidden reduction leaves at least $b$ vertices.

Remove from $C$ the separator vertices with no neighbour in $C\setminus D$, and call the resulting set $C'$. Thus $C'=(C\setminus D)\cup\dd_G(C\setminus D)$. In $H=R_{C',D}$, Lemma~\ref{lem:cancel} gives nonpositive density and $4$-lightness. Each root of $H$ has a neighbour in $C\setminus D$. As $B\subseteq D$, this implies that every root of $H$ belongs to $A$. Consequently $(A\cap D,B)$ is a root 5-separation of $H$. Theorem~\ref{thm:inputs}(3) gives a weak template in its right-hand side. There are at most four original roots in $H$, so Lemma~\ref{lem:nonfull} supplies a reducible fragment $Y$ of $H$. Its boundary and density are unchanged in $G$, since it avoids the roots of $H$ and has no neighbour in $C\setminus D$. Also $|V(G)\setminus Y|\ge|C'|\ge b$. This contradicts $b$-irreducibility.
\end{proof}
\needroom{6\baselineskip}
\begin{lemma}\label{lem:consistency}
\begin{enumerate}
\item If $(A,B)$ is a 5-separation and $R_{A,B}$ is quite heavy, then $R_{B,A}$ is $4$-light.
\item If $(C,D)$ has order at most four and $Y$ is a quite heavy 5-fragment with $Y\cup\dd Y\subseteq C$, then $|C\setminus D|\ge3$, $\rho_4(R_{C,D})\le0$, and $R_{C,D}$ is $4$-light.
\end{enumerate}
\end{lemma}
\begin{proof}
The conclusions are those of~\cite[Lemma~5.4 and Corollary~5.5]{DNR26}. Their use of irreducibility is replaced here by Lemma~\ref{lem:enclosing}. The complement in that application has at least six vertices, so either value of $b$ is allowed.

For (1), suppose the opposite side contains a positive $(\le4)$-fragment $L\subseteq A\setminus B$. Then $R_{A,B}$ is $4$-light, since a violating fragment in $B\setminus A$ would form a dense small bifragment with $L$. The complete opposite side of $L$, rooted on $\dd L$, has nonpositive density and is $4$-light for the same reason. Its other side has closed interior $L\cup\dd L$ of size at least six. Since $B$ is contained in this opposite side, Lemma~\ref{lem:enclosing} is contradicted.

For (2), quite heaviness implies $|Y|\ge2$, and hence $|C\setminus D|\ge|Y|+5-4\ge3$. Let $M\subseteq C\cap D$ consist of vertices with no neighbour in $D\setminus C$, and put $D'=D\setminus M$. This does not change the density of the right-hand side. Every vertex of $C\cap D'$ has a neighbour outside $C$, and hence cannot belong to $Y$. Thus $Y\subseteq C\setminus D'$. The complete opposite side of $Y$ is $4$-light by (1), and $(C\setminus Y,D')$ is a root separation of it of order at most four. Its density is therefore nonpositive, proving $\rho_4(R_{C,D})\le0$.

Finally, if $(I,J)$ is a small root separation of $R_{C,D}$, then $(I\cup C,J)$ is a separation of $G$ of order at most four with $Y\cup\dd Y\subseteq I\cup C$. The density conclusion just proved applies to it. This shows that $R_{C,D}$ is $4$-light.
\end{proof}
\needroom{6\baselineskip}
\begin{lemma}\label{lem:orienting}
Suppose that $\rho_4(G)=-d$, where $d\in\{8,9\}$, and additionally that $G$ contains no $K_5$ subgraph when $d=9$. Every 5-separation has exactly one quite heavy side. If such a side is $4$-light, it has a rooted $K_5$ minor.
\end{lemma}
\begin{proof}
The corresponding calculation in~\cite[Lemmas~5.7 and~5.8]{DNR26} uses global density $-7$ and a side-density sum at least $m+3$. Here the sum is $m+10-d$, and Lemma~\ref{lem:joining} handles the smaller threshold when both side densities are at least two. The extra $K_5$ exclusion for $d=9$ will ensure that the sum is still at least two.

Let the root set of a 5-separation have $m$ missing edges, and let its side densities be $r_1,r_2$. Counting edges gives
\begin{equation}\label{eq:sides}
r_1+r_2=m+10-d.
\end{equation}
When $d=9$, the absence of $K_5$ gives $m\ge1$. If neither side is quite heavy, then $r_1,r_2\le1$, and equality in the resulting bound forces $r_1=r_2=1$ and $m=d-8$. Each side then contains a vertex adjacent to all five roots. Together with the roots these vertices give $\Ke$. The only possible missing edges are the edge between them and, when $d=9$, one root edge.

If both sides are quite heavy, Lemma~\ref{lem:consistency}(1) makes both $4$-light. If one has density one, use Theorem~\ref{thm:inputs}(3) there. The other side has density $m+9-d\ge m$ and supplies all $m$ missing root edges by Theorem~\ref{thm:inputs}(1). The weak model with the root clique gives $\Ke$. If both densities are at least two, apply Lemma~\ref{lem:joining}. Thus exactly one side is quite heavy.

The other density is at most one, so the heavy side has density at least $m+9-d\ge m$. For $0\le m\le10$, the table defining $\T_r$ shows that $\T_r$ contains every $m$-edge graph whenever $r\ge m$. If the heavy side is $4$-light, it therefore supplies all missing root edges. Retaining existing root edges gives its rooted $K_5$ model.
\end{proof}
\needroom{6\baselineskip}
\begin{lemma}\label{lem:noncrossing}
Under the hypotheses of Lemma~\ref{lem:orienting}, there is no dense $(\le5)$-bifragment for which every 5-boundary member is quite heavy.
\end{lemma}
\begin{proof}
We use the noncrossing argument of~\cite[Lemma~5.9 and Corollary~5.10]{DNR26}. The rooted clique now comes from Lemma~\ref{lem:orienting} at density $-d$, and the only additional size check is for an enclosing reduction when $b=4$.

\smallskip\noindent\emph{Two quite heavy 5-fragments.}
Suppose that $Y,Z$ are nonadjacent 5-fragments and both are quite heavy. By Lemma~\ref{lem:consistency}(1), each associated rooted graph is $4$-light. It is a rooted side inside the complete opposite side of the other fragment. We claim that their closed neighbourhoods are joined by five vertex-disjoint paths. Otherwise Menger's theorem gives a separation $(C,D)$ of order at most four with
\[
Y\cup\dd_GY\subseteq C,\qquad Z\cup\dd_GZ\subseteq D.
\]
Lemma~\ref{lem:consistency}(2) gives $|C\setminus D|\ge3$, nonpositive density and $4$-lightness on the right.

For $b=3$, these facts suffice to apply Lemma~\ref{lem:enclosing}, contradicting the heavy side on $Z$. For $b=4$ we also need
\begin{equation}\label{eq:closedinterior}
|(C\setminus D)\cup\dd_G(C\setminus D)|\ge4.
\end{equation}
Put $E=C\setminus D$ and $Q=C\cap D$. If~\eqref{eq:closedinterior} fails, $|E|\ge3$ forces $|E|=3$ and $\dd_GE=\varnothing$. In particular there is no edge from $E$ to $Q$. Quite heaviness implies $|Y|\ge2$, so $|C|\ge|Y|+5\ge7$. On the other hand, $|C|=|E|+|Q|\le3+4=7$. Equality gives
\[
|C|=7,\qquad |Q|=4,\qquad |Y|=2,
\qquad C=Y\cup\dd_GY.
\]
If both vertices of $Y$ lie in $E$, their boundary has size at most one. If both lie in $Q$, it has size at most two. Here no vertex outside $C$ can be a neighbour of $Y$, and there are no edges between the two blocks. Both possibilities contradict $|\dd_GY|=5$. If instead one vertex lies in each block, they are nonadjacent and have at most two and three boundary neighbours, respectively. Consequently
\[
\rho_4(G,Y)\le0+(2+3)-4\cdot2=-3,
\]
contrary to positive density. This proves~\eqref{eq:closedinterior}, and Lemma~\ref{lem:enclosing} again excludes the separator.

Normalize the five paths to run from $\dd_GY$ to $\dd_GZ$, with interiors avoiding both closed neighbourhoods. A shared boundary vertex uses a zero-length path. Each boundary has five vertices, so all its vertices are used exactly once. Apply Theorem~\ref{thm:inputs}(3) to $R_Y$ and Lemma~\ref{lem:orienting} to $R_Z$. Transfer the rooted clique along the paths and combine it with the weak template on $Y$. The interiors of these models are disjoint, so they give a $\Ke$ model, a contradiction.

Let $(Y,Z)$ be a bifragment as in the statement. Both boundaries cannot have size at most four, by $4$-bilightness. If both have size five, the preceding case applies. Otherwise a positive fragment of boundary at most four lies in the complete opposite side of a quite heavy 5-fragment. This contradicts Lemma~\ref{lem:consistency}(1).
\end{proof}
\needroom{6\baselineskip}
\begin{lemma}\label{lem:contractstable}
Under the hypotheses of Lemma~\ref{lem:orienting}, contracting any edge of $G$ preserves $4$-bilightness.
\end{lemma}
\begin{proof}
We isolate the contraction step from the proof of~\cite[Lemma~6.2]{DNR26}. That proof uses it for a low-triangle edge in a minimal counterexample at density $-7$. Here we prove preservation for every edge, using $b$-irreducibility and Lemma~\ref{lem:noncrossing} at density $-d$. No restriction on the number of triangles is needed for this step.

Let $J=G/uv$, with contraction vertex $w$. Suppose that $J$ has a dense $(\le4)$-bifragment $(Y,Z)$, chosen to minimize $|Y|+|Z|$. Each associated rooted graph is $4$-light. Indeed, a positive fragment with boundary smaller than that rooted graph's root set would be a smaller positive $(\le4)$-fragment in $J$, still nonadjacent to the other member.

\smallskip\noindent\emph{Case 1. The contraction vertex lies in an interior.}
Suppose $w\in Y$. The case $w\in Z$ is symmetric. Theorem~\ref{thm:inputs}(2) supplies a boundary clique or dart in $R_Y^J$. Expand $w$ and put
\[
Y'=(Y\setminus\{w\})\cup\{u,v\}.
\]
Its boundary in $G$ is exactly $\dd_JY$, so has size at most four. Expanding $w$ within whichever branch set contains it lifts the rooted model to $R_{Y'}^G$. The edge $uv$ preserves connectivity. Also $|Y'|\ge2$.

The set $Z$ has no neighbour $w$ in $J$, since the two interiors are nonadjacent. Thus it has no neighbour $u$ or $v$ in $G$, and its boundary and positive density are unchanged. The sets $Y',Z$ are nonadjacent, so $4$-bilightness of $G$ forces $\rho_4(G,Y')\le0$. Moreover,
\[
Z\cup\dd_GZ\subseteq V(G)\setminus Y',\qquad
|Z\cup\dd_GZ|\ge6,
\]
where the last inequality is the positive-fragment estimate from Lemma~\ref{lem:stability}. Hence $Y'$ is reducible and leaves at least $6\ge b$ vertices. This contradicts $b$-irreducibility.

\smallskip\noindent\emph{Case 2. The contraction vertex lies in neither interior.}
Both $Y,Z$ survive as the same nonempty, disjoint, nonadjacent vertex sets in $G$. Expanding a boundary vertex increases its boundary size by at most one. Each remains positive and has boundary at most five. For $T\in\{Y,Z\}$ write
\[
c_T=|N_G(u)\cap N_G(v)\cap T|.
\]
Exactly these pairs of incident edges were merged by contraction, so
\[
\rho_4(G,T)=\rho_4(J,T)+c_T>0.
\]
If $|\dd_GT|=5$, then $w\in\dd_JT$, $|\dd_JT|=4$, and both $u,v$ belong to $\dd_GT$. If $\rho_4(G,T)\ge2$, this 5-rooted side is quite heavy by definition. If its density is one, positivity and integrality force $\rho_4(J,T)=1$ and $c_T=0$. No vertex of $T$ then meets both $u$ and $v$, and in particular none meets all five boundary roots. This also makes the side quite heavy. The resulting dense $(\le5)$-bifragment in $G$ contradicts Lemma~\ref{lem:noncrossing}.
\end{proof}

\needroom{10\baselineskip}
The next two lemmas concern vertices of degree at most seven.
\begin{lemma}\label{lem:small}
Let $L$ have between five and seven vertices and minimum degree at least four. For every 5-set $S\subseteq V(L)$ it has an $S$-rooted $K_5^-$ minor. For every set of at most four prescribed roots it has the corresponding rooted clique minor.
\end{lemma}
\begin{proof}
This is the local argument of~\cite[Lemma~6.4 and Corollary~6.5]{DNR26}. We include it to isolate its hypotheses. For five vertices, $L$ is complete. For six vertices its complement is a matching. If $uv$ is a missing edge on $S$, the outside vertex is adjacent to both $u,v$. Absorb it into $u$. At most one missing root edge remains.

Suppose $|V(L)|=7$ and let $z_1,z_2$ be outside $S$. The complement has maximum degree two. Any vertex of complement degree two within $S$ is adjacent in $L$ to both $z_i$. If the complement on $S$ is a 5-cycle $v_1\ldots v_5$, absorb $z_1$ into $v_1$ and $z_2$ into $v_3$. The same choice works if it contains a triangle or a 4-cycle on consecutively labelled vertices. At most one missing root edge remains. Otherwise that complement is a forest. If it has four edges, it is $v_1\ldots v_5$. In this case, choose the outside labels so that $v_1z_1$ is an edge of $L$, and absorb $z_1$ into $v_1$ and $z_2$ into $v_3$. If it has at most three edges and contains a path $v_1v_2v_3$, absorb both outside vertices into $v_2$. Each endpoint meets at least one outside vertex in $L$, so only one root nonedge can remain.

The remaining complement on $S$ is a matching. Only the case of two edges $ab,cd$ needs consideration. If an outside vertex meets both ends of either pair in $L$, absorb it into an end. Otherwise the complement contains matchings from $\{z_1,z_2\}$ onto $\{a,b\}$ and onto $\{c,d\}$. Its degree bound gives $z_1z_2\in E(L)$. Relabel so that $az_1,bz_2\in E(L)$, and contract these two edges. Again at most the nonedge $cd$ remains.

For four prescribed roots, extend to five and use the first assertion. If both ends of the missing edge are prescribed, absorb the fifth branch set into one end. Otherwise discard it. This gives a rooted $K_4$. Restriction yields the smaller rooted cliques.
\end{proof}
\needroom{6\baselineskip}
\begin{lemma}\label{lem:neighbourhood}
Suppose additionally that $|V(G)|\ge7$, every edge of $G$ is in at least four triangles, and $v$ has degree at most seven. Put $Q=N[v]$, and let $C_i$ be the components of $G-Q$. Then
\[
|\dd C_i|\le4,\qquad \rho_4(G,C_i)\le0,\qquad \rho_4(G)\le e(G[Q])-4|Q|.
\]
\end{lemma}
\begin{proof}
This is the neighbourhood argument of~\cite[Lemma~6.6]{DNR26}, with its local degree hypotheses stated separately. We use $b$-irreducibility only to exclude an isolated vertex and to exclude a boundary-clique reduction whose complement has at least six vertices. Thus the argument works for both $b=3$ and $b=4$. An isolated vertex would be reducible, since $|V(G)|\ge7$. Every other vertex is incident with an edge in at least four triangles and hence has degree at least five. Let $L=G[N(v)]$. It has between five and seven vertices and minimum degree at least four. Since $v$ has no neighbour outside $Q$, every $\dd C_i$ is contained in $N(v)$. If some $C_i$ has at least five boundary vertices, choose five of them as roots in Lemma~\ref{lem:small}. Contract $C_i$, retain $v$, and combine with the rooted $K_5^-$. The only possible missing edges are one root edge and the edge between $v$ and the contracted component, which are independent. This gives $\Ke$, a contradiction.

If $\rho_4(G,C_i)>0$, put $T=\dd C_i$ and $W=V(G)\setminus(C_i\cup T)$. This set contains $v$, is nonadjacent to $C_i$, and has boundary contained in $T$. Thus $\rho_4(G,W)\le0$. Apply Lemma~\ref{lem:small} to $L$ with root set $T$, and retain only the root branch sets corresponding to $\dd_GW$. Each retained branch set avoids the other vertices of $T$, and every vertex of $L-T$ lies in $W$. The resulting boundary clique model therefore lies in $G[W\cup\dd_GW]$. Hence $W$ is reducible, and its complement $C_i\cup T$ has at least six vertices. This contradicts $b$-irreducibility. Finally the components are pairwise nonadjacent, so
\begin{equation}\label{eq:densityneighbourhood}
\rho_4(G)=e(G[Q])-4|Q|+\sum_i\rho_4(G,C_i)\le e(G[Q])-4|Q|.\qedhere
\end{equation}
\end{proof}

\section{Proofs of the main theorems}\label{sec:proofs}
\subsection{Proof of Theorem~\ref{thm:aux}}\label{subsec:aux}
Suppose Theorem~\ref{thm:aux} is false and choose a counterexample $G$ with the fewest vertices. By~\eqref{eq:known},
\begin{equation}\label{eq:auxdensity}
\rho_4(G)=-8,\qquad |V(G)|\ge7.
\end{equation}
The order bound follows also from $\binom{n}{2}\ge4n-8$ and $n\ge3$. Reducing an inclusionwise maximal reducible fragment preserves $4$-bilightness, does not decrease density, and leaves at least three vertices. Minimality therefore shows that $G$ is $3$-irreducible.

Apply Lemmas~\ref{lem:orienting}--\ref{lem:contractstable} with $b=3$ and $d=8$. If an edge $uv$ has at most three common neighbours, then $J=G/uv$ is $4$-bilight and
\[
\rho_4(J)=\rho_4(G)+3-|N(u)\cap N(v)|\ge-8.
\]
It is a smaller counterexample, a contradiction. Every edge is thus in at least four triangles. Since $G$ is $3$-irreducible, it has no isolated vertex. Therefore $G$ has minimum degree at least five. Its average degree is less than eight, so choose a vertex $v$ of degree between five and seven, and put $Q=N[v]$. Lemma~\ref{lem:neighbourhood} gives
\[
e(G[Q])\ge4|Q|-8.
\]
For $|Q|=6$ this exceeds the number of possible edges. For $|Q|=7$, it gives at least twenty edges and hence an $\Ke$ subgraph. It remains that $|Q|=8$ and $e(G[Q])\ge24$.

The complement on $N(v)$ has maximum degree at most two and at most four edges. If it has at most three edges and a vertex of degree two, delete such a vertex. This leaves at most one nonedge. Otherwise its edges form a matching, and deleting an endpoint of a matching edge leaves at most two independent nonedges. If there are no nonedges, delete any vertex of $N(v)$. Each case gives $\Ke$ on seven vertices. If the complement $F$ on $N(v)$ has four edges, its nontrivial components are among the six types below. Isolated vertices are omitted, and path vertices are labelled consecutively. The operations are performed in $G[Q]$. The last column records its remaining nonedges.
\begin{center}
\small
\begin{tabular}{lll}
\toprule
Nontrivial components of $F$&Operation in $G[Q]$&Remaining nonedges\\
\midrule
$C_3\sqcup K_2$&Delete a triangle vertex&$2K_2$\\
$P_5=x_1\cdots x_5$&Delete $x_3$&$x_1x_2,\ x_4x_5$\\
$P_4\sqcup K_2$&Delete $x_2$ of $P_4$&$x_3x_4$ and the $K_2$ edge\\
$P_3\sqcup2K_2$&Delete the $P_3$ centre&$2K_2$\\
$2P_3$&Contract the two centres&None\\
$C_4$&Retain for the next step&$C_4$\\
\bottomrule
\end{tabular}
\end{center}
The list is exhaustive because a graph of maximum degree two is a union of paths and cycles, and four disjoint edges would require eight vertices in $N(v)$. In the $2P_3$ row the two centres are adjacent in $G$. Their sets of nonneighbours are disjoint. The contracted vertex is therefore adjacent to all remaining vertices, which also form a clique, giving $K_7$. Every other row except $C_4$ gives a $\Ke$ subgraph after the indicated deletion. Thus the only remaining possibility is
\begin{equation}\label{eq:exception}
G[Q]=K_4\vee(2K_2).
\end{equation}
Write $K$ for the 4-vertex clique and $P,T$ for the two pairs in~\eqref{eq:exception}. Each pair induces an edge, and there are no edges between the pairs. Thus the complement on $P\cup T$ is $K_{2,2}$, a 4-cycle. The other four vertices of $Q$ are universal in $G[Q]$. Each component $C_i$ of $G-Q$ has nonpositive density by Lemma~\ref{lem:neighbourhood}. Since $e(G[Q])=24$, equality in~\eqref{eq:densityneighbourhood} and~\eqref{eq:auxdensity} gives $\rho_4(G,C_i)=0$ for every $i$. No component meets both $P$ and $T$. Otherwise choose a path from $p\in P$ to $t\in T$ with its internal vertices in that component. Absorb all those internal vertices into $p$, keeping every vertex of $Q$ in a distinct branch set. This adds the edge $pt$ to the retained graph on $Q$. Any other new edges can be deleted. The complement $C_4$ would become a 3-edge path. Deleting an internal vertex of this path gives $K_7^-$.

Let $Y$ consist of $P$ and all components meeting $P$, and let $Z$ consist of $T$ and all components meeting $T$. No component was assigned to both sets. Distinct components of $G-Q$ have no edges between them, there is no edge between $P$ and $T$, and any component adjacent to $P$ or $T$ was included in the corresponding set. Hence $Y,Z$ are disjoint and nonadjacent, and every neighbour outside either set belongs to $K$. The zero densities of the added components give
\[
\begin{aligned}
\rho_4(G,Y)&=1+8-4\cdot2+\sum_{C_i\text{ meeting }P}\rho_4(G,C_i)=1,\\
\rho_4(G,Z)&=1+8-4\cdot2+\sum_{C_i\text{ meeting }T}\rho_4(G,C_i)=1.
\end{aligned}
\]
Here the first term counts the edge within the pair, and the second its eight edges to $K$. This dense small bifragment contradicts $4$-bilightness and completes the proof of Theorem~\ref{thm:aux}.

\subsection{Proof of Theorem~\ref{thm:bilight}}\label{subsec:critical}
Suppose Theorem~\ref{thm:bilight} is false, and choose a counterexample $G$ with the fewest vertices. It has neither an $\Ke$ minor nor a $K_6$ subgraph. By Theorem~\ref{thm:aux},
\begin{equation}\label{eq:criticaldensity}
\rho_4(G)=-9,\qquad |V(G)|\ge7.
\end{equation}
Orders four and five cannot meet the edge bound, and at order six the only graph meeting it is $K_6$.
\needroom{6\baselineskip}
\begin{lemma}\label{lem:criticalirreducible}
The graph $G$ is $4$-irreducible.
\end{lemma}
\begin{proof}
Suppose otherwise, and choose an inclusionwise maximal reducible fragment $Y$ subject to $|V(G)\setminus Y|\ge4$. Put $Q=\dd_GY$ and let $J$ be the reducent. By Lemma~\ref{lem:stability}(2), $J$ is $4$-bilight. It is a minor of $G$ and therefore has no $\Ke$ minor.

Let $\varepsilon=0$ for a clique reduction and $\varepsilon=1$ for a dart reduction, and let $a$ be the number of edges added between surviving original vertices of $Q$. A dart introduces one new vertex and its four incident edges. A clique reduction introduces no vertex. Consequently
\[
|V(J)|=|V(G)|-|Y|+\varepsilon
\]
and
\[
e(J)=e(G)-e(G[Y])-e_G(Y,Q)+a+4\varepsilon.
\]
Subtracting four times the vertex count gives
\[
\rho_4(J)-\rho_4(G)=-\rho_4(G,Y)+a\ge0.
\]
Here both $-\rho_4(G,Y)$ and $a$ are nonnegative integers.

Since $|V(J)|\ge|V(G)\setminus Y|\ge4$, Theorem~\ref{thm:aux} applies to $J$. Its absence of a $\Ke$ minor implies $\rho_4(J)<-8$, and integrality gives $\rho_4(J)\le-9$. Combining this with~\eqref{eq:criticaldensity} and the preceding inequality yields
\begin{equation}\label{eq:zero}
\rho_4(J)=\rho_4(G)=-9,\qquad
\rho_4(G,Y)=0,\qquad a=0.
\end{equation}
The last two equalities follow because their nonnegative integer contributions sum to zero. The equality $a=0$ means that the reduction has added no edge between surviving original vertices. This is the property needed below.

In particular, the graph induced by the surviving original vertices of $J$ is exactly $G-Y$. In a clique reduction, every vertex of $J$ is such a vertex, so a $K_6$ subgraph of $J$ would already be a $K_6$ subgraph of $G$. In a dart reduction the only other vertex has degree four, whereas every vertex in a $K_6$ subgraph has at least five neighbours. It cannot belong to a 6-clique, and any 6-clique avoiding it would again be present in $G$. Therefore $J$ has no $K_6$ subgraph.

Finally, $J$ has fewer vertices than $G$. A clique reduction removes the nonempty set $Y$, while a dart reduction replaces at least two vertices by one. Hence $J$ is a smaller $4$-bilight graph on at least four vertices with $\rho_4(J)=-9$, no $\Ke$ minor and no $K_6$ subgraph. This contradicts the minimal choice of $G$.
\end{proof}
\needroom{6\baselineskip}
\begin{lemma}\label{lem:noK5}
$G$ has no $K_5$ subgraph.
\end{lemma}
\begin{proof}
Suppose that $C$ is a 5-vertex clique. We first show that $(G,C)$ is $4$-light. If not, choose a positive $(\le4)$-fragment $Y$ avoiding $C$, put $T=\dd_GY$, and write $k=|T|\le4$ and $B=G-Y$. The set $Y\cup T$ has at least six vertices by the positive-fragment estimate in the proof of Lemma~\ref{lem:stability}.

Suppose first that $B$ has $k$ vertex-disjoint paths from $T$ to $C$, with zero-length paths allowed for vertices in $T\cap C$. Truncate each path at its first vertex of $C$. Each vertex of $T$ is the initial vertex of exactly one path, and the endpoints in $C$ are distinct. Use each path as one root branch set. The edges between their distinct endpoints in $C$ give a $T$-rooted $K_k$ model in $B$. Every vertex of $T$ starts its own path, so no retained branch set can pass through a different vertex of $T$.

Set $W=V(G)\setminus(Y\cup T)$. Since $C\cap Y=\varnothing$ and $|C|=5>k$, $W$ is nonempty. It is nonadjacent to $Y$ and has boundary contained in $T$. Thus $4$-bilightness forces $\rho_4(G,W)\le0$. Restrict the clique model to the roots in $\dd_GW$ if necessary. Its retained branch sets avoid all other vertices of $T$, so it lies in $G[W\cup\dd_GW]$. Hence $W$ is reducible and its complement $Y\cup T$ has at least six vertices, contrary to Lemma~\ref{lem:criticalirreducible}. If $k=0$, the same conclusion follows by deleting the component set $W$. The boundary clique is empty.

Otherwise, let $q<k$ be the maximum number of disjoint $T$--$C$ paths. By Menger's theorem, choose a separation $(I,J)$ of $B$ of order $q$ with $T\subseteq I$ and $C\subseteq J$. Put $Q=I\cap J$ and $U=J\setminus I$. At least $5-q$ vertices of $C$ lie in $U$, so $U$ is nonempty. Also $U\cap T=\varnothing$ and $\dd_GU\subseteq Q$. The separation gives this in $B$, and no vertex of $U$ can meet $Y$ because all neighbours of $Y$ outside $Y$ lie in $T$.

A maximum family of $q$ disjoint paths meets every vertex of $Q$, each on a different path. Indeed, every path must cross $Q$, and there are only $q$ separator vertices. Each path contains exactly one separator vertex. Indeed, if one contained two, the other $q-1$ disjoint paths could not all cross the remaining $q-2$ vertices. After its separator vertex, a path stays in $J$, since returning from $I\setminus J$ would require a second separator vertex. Its tail to $C$ is therefore a branch set in $J$ with exactly one root in $Q$. The endpoints in $C$ are distinct and mutually adjacent, so these tails give a $Q$-rooted clique model. Restrict it to the actual boundary $\dd_GU$. The other separator vertices belong to discarded root branch sets. This gives a boundary clique model in $G[U\cup\dd_GU]$. When $q=0$, the empty boundary clique has the same conclusion.

The set $U$ is nonadjacent to the positive fragment $Y$, so $\rho_4(G,U)\le0$ by $4$-bilightness. Its complement contains $Y\cup T$ and therefore at least six vertices. Thus $U$ is another forbidden reducible fragment. This proves that $(G,C)$ is $4$-light.

Now
\[
\rho_4((G,C))=\rho_4(G)+20-e(G[C])=-9+20-10=1.
\]
No vertex outside $C$ meets all five vertices of $C$, since such a vertex would form a $K_6$ subgraph with $C$. Hence $(G,C)$ is quite heavy, and Theorem~\ref{thm:inputs}(3) supplies a weak $\W$ model. Retain the already present root clique on $C$. The resulting seven branch sets miss only the edge between the two nonroots, or at most two independent root--nonroot edges. They therefore contain a $\Ke$ model, the required contradiction.
\end{proof}
We can now apply the structural lemmas with $b=4$, $d=9$. In particular all edge contractions preserve $4$-bilightness. If an edge $uv$ has at most three common neighbours, then $J=G/uv$ satisfies $\rho_4(J)\ge-9$. It also has no $K_6$ subgraph. A new 6-clique using the contraction vertex would have five other vertices that already form a $K_5$ in $G$, contrary to Lemma~\ref{lem:noK5}. A 6-clique not using it would already be in $G$. As $|V(J)|\ge6$, minimality is contradicted.

Every edge therefore lies in at least four triangles. Irreducibility excludes isolated vertices, so the minimum degree is at least five. The average degree is less than eight. Choose $v$ with $5\le d(v)\le7$ and put $Q=N[v]$. By Lemma~\ref{lem:neighbourhood},
\begin{equation}\label{eq:criticalQ}
e(G[Q])\ge4|Q|-9.
\end{equation}
If $|Q|=6$, this forces $G[Q]=K_6$. If $|Q|=7$, it has at most two missing edges, and deleting at most two suitable vertices leaves a $K_5$. Both cases contradict Lemma~\ref{lem:noK5}.

Thus $|Q|=8$ and $e(G[Q])\ge23$. Let $F$ be the complement of $G[N(v)]$. It has seven vertices, at most five edges and maximum degree at most two. It has no independent set of size four, since such a set together with $v$ would be a $K_5$ in $G$. If $F$ were bipartite, one part would have size at least four. Hence it has an odd cycle. A 7-cycle exceeds the edge budget. A 5-cycle uses all five edges and leaves two isolated vertices, again giving an independent 4-set. The odd cycle is therefore a triangle. The triangle is a component because $\Delta(F)\le2$. On the other four vertices there are at most two edges. With at most one edge there is an independent 3-set. With two edges sharing an endpoint, their two other endpoints and the fourth vertex form an independent 3-set. Either 3-set, together with one triangle vertex, would be an independent 4-set of $F$. Thus the two edges on the remaining vertices must be disjoint. Consequently
\[
F=K_3\sqcup2K_2.
\]
Label the triangle $abc$ and the two other edges $de,fg$. The edge $ad$ is present in $G[Q]$. Contract it. The new vertex is adjacent to all six remaining vertices, and the only remaining nonedges are $bc$ and $fg$, which are independent. The resulting minor is $\Ke$, the final contradiction. This proves Theorem~\ref{thm:bilight} and hence Theorem~\ref{thm:main}.

\section*{Acknowledgements}
This work was supported by the National Natural Science Foundation of China under Grant No.~12161073, the Guangdong Basic and Applied Basic Research Foundation under Grant No.~2026A1515012764, the GDUPT Talent Recruitment Project (No.~2024rcyj1007), and the Scientific Research Funds at China University of Geosciences (Wuhan) (Project No.~2026039).

\section*{Declaration of generative AI and AI-assisted technologies}
OpenAI Codex contributed ideas and suggested a proof strategy for Theorem~\ref{thm:five}. It was also used to assist with the organization and editing of the manuscript and the preparation of reproducible diagram code. All theoretical results and proofs have been checked by the authors.

\end{document}